\pdfoutput=1
\documentclass[11pt,a4paper]{article}
\usepackage[a4paper,margin=25mm]{geometry}
\usepackage{lmodern}
\usepackage[T1]{fontenc}
\usepackage[utf8]{inputenc}
\usepackage{amsmath,amssymb,amsthm,mathtools}
\usepackage{booktabs,microtype}
\usepackage[shortlabels]{enumitem}
\usepackage[round,authoryear]{natbib}
\usepackage{authblk}
\usepackage[unicode]{hyperref}
\hypersetup{hidelinks,pdftitle={Total Path Length in Power-Weight Recursive Trees: Martingale Limits and Global Fluctuations},pdfauthor={Marek Gałązka and Hanna Wdowicka},pdfsubject={Probability; combinatorics},pdfkeywords={random recursive trees, total path length, martingales, limit theorems, probabilistic combinatorics, analysis of algorithms}}
\newtheorem{theorem}{Theorem}[section]
\newtheorem{proposition}[theorem]{Proposition}
\newtheorem{lemma}[theorem]{Lemma}
\newtheorem{corollary}[theorem]{Corollary}

\newtheorem{remark}[theorem]{Remark}
\newcommand{\E}{\mathbb{E}}
\newcommand{\Pp}{\mathbb{P}}
\newcommand{\Var}{\operatorname{Var}}

\newcommand{\F}{\mathcal{F}}
\newcommand{\R}{\mathbb{R}}

\title{\bfseries Total Path Length in Power-Weight Recursive Trees:\\
Martingale Limits and Global Fluctuations}
\author[1]{Marek Gałązka}
\author[2]{Hanna Wdowicka}
\affil[1]{Faculty of Mathematics and Computer Science, Adam Mickiewicz University, Poznań, Poland}
\affil[2]{Department of Statistics, Poznań University of Economics and Business, Poland}
\date{September 2026}

\begin{document}
\maketitle
\begin{abstract}
We study total path length in recursive trees with positive deterministic attachment weights.
Writing $W_n=\sum_{i=1}^n w_i$ and $p_n=w_n/W_n$, we obtain exact martingale-innovation identities and a variance recurrence.
Under the condition $p_n=O(n^{-1})$, centered total path length divided by $n$ converges almost surely and in $L^2$ to a nondegenerate random variable, and its variance is asymptotic to a positive constant times $n^2$.
No polynomial asymptotic for $W_n$ is required.
For power weights $w_i=i^\alpha$, the same argument applies to every real $\alpha$, including the critical and summable regimes beyond the positive-power cumulative-weight assumptions of existing profile theory.
The expected average depth is logarithmic for $\alpha>-1$, iterated logarithmic for $\alpha=-1$, and bounded for $\alpha<-1$, while the global fluctuation scale remains linear throughout.
In the summable regime we identify the random limit through the weighted depths of the infinite tree.
The uniform case recovers the classical variance coefficient $2-\pi^2/6$.
For linear weights we evaluate the coefficient as $8-2\pi^2/3$.
Although this tree and a random binary search tree have identical insertion-depth marginals and expected total path length, their asymptotic variance coefficients differ by one.
This gives an explicit comparison of global dependence that is invisible in individual depth distributions.
\end{abstract}


\smallskip


\section{Introduction}\label{sec:intro}
Random recursive trees are standard probabilistic models of incrementally generated rooted structures.
Vertices arrive one at a time, and every new vertex stores one link to an earlier vertex.
The uniform model has a long history in combinatorial probability and the analysis of algorithms; see
\citet{Szymanski1990}, \citet{SmytheMahmoud1995}, and \citet{Drmota2009}.
The total path length
\[
T_n:=\sum_{i=1}^n D_i,
\]
where $D_i$ is the depth of vertex $i$, measures how far vertices lie from the root in aggregate.
Its expectation is determined by individual depth means, whereas its fluctuations depend on the joint growth of the tree.
We give a direct martingale-innovation analysis that applies without a regular cumulative-weight asymptotic and makes these fluctuations explicit.

Total path length also has a direct cost interpretation.
Tracing parent pointers from vertex $i$ to the root takes exactly $D_i$ pointer traversals.
Thus $T_n$ is the cumulative cost of performing this operation once from every vertex,
and $T_n/n$ is its conditional average cost for a uniformly sampled stored vertex.

The classical uniform recursive tree already shows that global path length
contains information not visible from the depth of one insertion.
\citet{Mahmoud1991} established an almost-sure and quadratic-mean limit theorem
for its centered internal path length, and \citet{DobrowFill1999} developed
the limiting distribution in substantially greater detail.  Total path
length also plays a central role in the average-case analysis of binary search
trees and Quicksort; see \citet{Knuth1998} and
\citet{FlajoletSedgewick2009}.  In a random binary search tree the expected
internal path length equals
\[
2(n+1)H_n-4n,
\]
while its variance is asymptotic to
$(7-2\pi^2/3)n^2$.

The present paper concerns nonuniform recursive growth.  In a weighted
recursive tree, vertex $i$ receives a deterministic positive weight $w_i$ and
a later vertex chooses its parent with probability proportional to these
weights.  General weighted recursive trees have been studied from several
viewpoints; see, for example, \citet{BorovkovVatutin2006},
\citet{HiesmayrIslak2020}, and \citet{Senizergues2021}.  We focus on the
power-weight family
\[
w_i=i^\alpha,\qquad \alpha\in\R.
\]
This model interpolates between preference for recent vertices
($\alpha>0$), uniform attachment ($\alpha=0$), and preference for old
vertices ($\alpha<0$).

The linear rule $\alpha=1$ appears as a label-attraction model in
\citet{GalazkaEtAl2020}, where degree statistics and degree counts were
studied, and as the young-age preferential-attachment tree of
\citet{LyonMahmoud2020}, where insertion depth, outdegrees, and leaves were
analyzed.  \citet{LyonMahmoud2022} subsequently treated insertion depth in the
full power-weight family.  The Bernoulli representation underlying insertion
depth is implicit in the aggregation-tree coupling of
\citet[equation~(2.2)]{CurienHaas2017} and is explicit for arbitrary positive
deterministic weights in \citet[Corollary~8]{MaillerUribeBravo2019}; see also
the weighted-recursive-tree and general-affinity treatments in
\citet{Senizergues2021,NakataMahmoud2024}.  In particular, for
$\alpha>-1$ the typical insertion-depth scale is logarithmic, at
$\alpha=-1$ it is of order $\log\log n$, and for $\alpha<-1$ it remains on a
finite scale.  In the special young-age case $\alpha=1$, the insertion depth
has exactly the same distribution as the insertion depth in a random binary
search tree \citep{LyonMahmoud2020}.

Path-length functionals have also been studied in several other nonuniform or
weighted tree models.  Examples include weighted path lengths in increasing
trees \citep{KubaPanholzer2007}, internal path length in Hoppe trees
\citep{LeckeyNeininger2013}, and path length in weighted random $b$-ary
recursive trees \citep{MunsoniusRuschendorf2011}.  These uses of ``weighted''
refer to different tree families or to weights on path functionals; they do
not give the total path-length theory for deterministic label weights
$w_i=i^\alpha$ considered here.

\paragraph{Relation to the weighted-recursive-tree profile theory.}
For $\alpha>-1$, the weights $w_i=i^\alpha$ satisfy the assumptions of the
profile analysis of \citet{Senizergues2021}, with $\gamma=\alpha+1$.
In the section ``Height and profile of WRT'', the normalized profile transform
\[
  N_n(z)=n^{-1-\gamma(e^z-1)}\sum_{i=1}^n e^{zD_i}
\]
is shown to converge almost surely, locally uniformly on a complex
neighbourhood of zero, to a random analytic function.
Since $N_n'(0)=T_n/n-\gamma\log n$, differentiation gives an almost-sure
limit for this centered average depth. The convergence of
$(T_n-\mathbb{E}T_n)/n$ in this regime follows after accounting for the
deterministic constant in its mean asymptotics.
The weighted-depth martingale used below also has a direct antecedent in
the section ``Study of the Laplace transform of the weighted profile''.
With normalization starting at index $2$, differentiating
\[
  \mathcal{M}_n(z)=
  \frac{W_n^{-1}\sum_{i=1}^n w_i e^{zD_i}}
       {\prod_{i=2}^n\bigl(1+(e^z-1)p_i\bigr)}
\]
at zero, with $A_n=\sum_{i=1}^n w_iD_i$, yields
$\mathcal{M}_n'(0)=(A_n-\mathbb{E}A_n)/W_n$.

\paragraph{Contributions and scope.}
The model, the one-vertex Bernoulli representation, and the associated depth scales are known.
The profile theory described above also implies the almost-sure centered path-length limit for power weights with $\alpha>-1$.
We do not claim that limit or the underlying weighted-profile martingale as new.
The results below provide a common elementary proof without prescribing a cumulative-weight asymptotic, together with explicit second-order calculations.
More precisely:
\begin{enumerate}[label=(\roman*)]
\item For arbitrary positive deterministic weights, we derive the exact identity
\[
T_n-\E T_n=\sum_{j=2}^n\bigl(1+(n-j)p_j\bigr)\xi_j
\]
and an exact recurrence for $v_j=\E\xi_j^2$.
These formulas describe how each attachment innovation contributes to total path length and its variance.
\item If $p_n=O(n^{-1})$, then $(T_n-\E T_n)/n$ converges almost surely and in $L^2$ to a nondegenerate limit, with
\[
\Var(T_n)\sim n^2\sum_{j=2}^\infty p_j^2v_j.
\]
The hypothesis covers all power weights, including $\alpha=-1$ and $\alpha<-1$, and also bounded oscillating weights for which $W_n/n$ has no limit.
\item For $\alpha<-1$, the average depth has the explicit random limit
\[
\frac{T_n}{n}\longrightarrow 1+\frac{\sum_{i\ge1}i^\alpha D_i}{\zeta(-\alpha)}
\qquad\text{almost surely and in }L^2.
\]
Thus the bounded expected average depth in this regime does not imply convergence to a deterministic constant.
\item For $\alpha=1$, we solve the variance recurrence and obtain
\[
\Var(T_n)\sim\left(8-\frac{2\pi^2}{3}\right)n^2.
\]
The coefficient is larger by exactly one than that of a random binary search tree, although the insertion-depth marginals and expected total path length coincide.
\end{enumerate}
For completeness, we also obtain the exact expectation and its three power-weight regimes by summing the known one-vertex means,
recover the classical uniform-tree benchmark, and give a simple exact parent sampler for linear weights.
These supporting results place the fluctuation analysis in context.

Section~\ref{sec:model} introduces the model and recalls the depth representation.
Sections~\ref{sec:sampling} and \ref{sec:expectation} give the linear-weight sampler and the power-weight expectation asymptotics.
Sections~\ref{sec:martingale} and \ref{sec:limit} establish the exact innovation identities and the general limit theorem.
Sections~\ref{sec:oldage} and \ref{sec:youngage} identify the summable-regime limit and evaluate the linear-weight variance coefficient.

\section{Weighted recursive trees and power weights}\label{sec:model}

\subsection{Definition}
Let $w_1,w_2,\ldots>0$ be deterministic weights and put
\[
W_n:=\sum_{i=1}^n w_i,\qquad p_j:=\frac{w_j}{W_j}.
\]
The weighted recursive tree is constructed on $\{1,\ldots,n\}$ from root $1$.
For $n\ge2$, vertex $n$ independently chooses a parent $K_n\in\{1,\ldots,n-1\}$ with
\begin{equation}\label{eq:parent-law}
\Pp(K_n=i)=\frac{w_i}{W_{n-1}},\qquad 1\le i<n.
\end{equation}
All trees are coupled using the same infinite sequence of independent parent choices.
Let $D_1=0$ and let $D_n$ be the depth of vertex $n$. Then
\begin{equation}\label{eq:depth-rec}
D_n=1+D_{K_n},\qquad n\ge2.
\end{equation}
The power-weight specialization is $w_i=i^\alpha$, $\alpha\in\R$, and is denoted by $T_n^{(\alpha)}$ when the tree itself is needed.
For these weights,
\begin{equation}\label{eq:pj}
p_j=\frac{j^\alpha}{\sum_{i=1}^j i^\alpha}.
\end{equation}
In either setting $p_1=1$, and total path length is
\begin{equation}\label{eq:Tn}
T_n:=\sum_{i=1}^nD_i.
\end{equation}

\subsection{Known insertion-depth representation}

The following representation is a direct consequence of the
ancestral-indicator representation in
\citet[Corollary~8]{MaillerUribeBravo2019}, which holds for arbitrary positive
weight sequences.  Indeed, the parent $K_n$ is sampled proportionally to the
weights in the tree on $n-1$ vertices, and $D_n=D_{K_n}+1$.  An implicit
precursor appears in the aggregation-tree coupling of
\citet[equation~(2.2)]{CurienHaas2017}.  Later treatments include
\citet{Senizergues2021,LyonMahmoud2022,NakataMahmoud2024}.  We include a short
derivation for notation and self-containedness because the parameters $p_j$
drive the entire path-length analysis; no novelty is claimed for this
representation.

\begin{proposition}[Bernoulli representation]\label{prop:bernoulli}
For every $n\ge2$,
\[
D_n\overset{d}=\sum_{j=1}^{n-1}B_j,
\]
where $B_1,B_2,\ldots$ are independent Bernoulli variables satisfying
$\Pp(B_j=1)=p_j$.  Equivalently,
\begin{equation}\label{eq:depth-pgf}
\E z^{D_n}
=
\prod_{j=1}^{n-1}(1-p_j+p_jz).
\end{equation}
Consequently,
\begin{align}
\E D_n&=\sum_{j=1}^{n-1}p_j,\label{eq:ED}\\
\Var(D_n)&=\sum_{j=1}^{n-1}p_j(1-p_j).\label{eq:VD}
\end{align}
\end{proposition}

\begin{proof}
Let $F_n(z):=\E z^{D_n}$ and define
\[
S_m(z):=\sum_{i=1}^m w_iF_i(z).
\]
Conditioning on $K_n$ in \eqref{eq:depth-rec} gives
\[
F_n(z)=z\frac{S_{n-1}(z)}{W_{n-1}}.
\]
Therefore
\[
S_n(z)
=
S_{n-1}(z)+w_nF_n(z)
=
S_{n-1}(z)\left(1+\frac{zw_n}{W_{n-1}}\right),
\]
and hence
\[
\frac{F_{n+1}(z)}{F_n(z)}
=
\frac{W_{n-1}+zw_n}{W_n}
=
1+(z-1)\frac{w_n}{W_n}
=
1-p_n+p_nz.
\]
Since $F_2(z)=z$ and $p_1=1$, induction yields
\eqref{eq:depth-pgf}.  The moment formulas follow immediately.
\end{proof}

\begin{remark}[Novelty boundary for depth]\label{rem:depth-boundary}
Proposition~\ref{prop:bernoulli} and the resulting one-vertex depth moments are
background.  The analysis below concerns the accumulated global functional
$T_n$, its martingale representation, and its fluctuations across insertion
times.
\end{remark}

\begin{lemma}[Asymptotics of the Bernoulli parameters]\label{lem:p-asym}
For power weights $w_n=n^\alpha$, the sequence $(p_n)$ has the following behavior.
\begin{enumerate}[label=(\roman*)]
\item If $\alpha>-1$ and $c_\alpha:=\alpha+1$, then for some
$\delta_\alpha>0$,
\begin{equation}\label{eq:p-supercrit}
p_n=\frac{c_\alpha}{n}+O(n^{-1-\delta_\alpha}),
\end{equation}
and in particular
\[
\sum_{n\ge1}\left|p_n-\frac{c_\alpha}{n}\right|<\infty.
\]

\item If $\alpha=-1$, then
\begin{equation}\label{eq:p-critical}
p_n=\frac{1}{nH_n},
\end{equation}
where $H_n=\sum_{k=1}^n1/k$.

\item If $\alpha<-1$, then
\begin{equation}\label{eq:p-old}
p_n\sim\frac{n^\alpha}{\zeta(-\alpha)}
\end{equation}
and $\sum_{n\ge1}p_n<\infty$.
\end{enumerate}
\end{lemma}

\begin{proof}
For $\alpha>-1$, comparison of the power sum with its integral gives
\[
W_n=\frac{n^{\alpha+1}}{\alpha+1}
\begin{cases}
1+O(n^{-1}),&\alpha\ge0,\\
1+O(n^{-(\alpha+1)}),&-1<\alpha<0.
\end{cases}
\]
Substitution in \eqref{eq:pj} gives \eqref{eq:p-supercrit}, with
$\delta_\alpha=\min\{1,\alpha+1\}$.  The error is summable.

For $\alpha=-1$, $W_n=H_n$, which gives \eqref{eq:p-critical}.
If $\alpha<-1$, then
$W_n\to\sum_{k\ge1}k^\alpha=\zeta(-\alpha)$, proving
\eqref{eq:p-old}; summability follows because $\alpha<-1$.
\end{proof}

\section{An exact linear-time generator for the young-age tree}
\label{sec:sampling}

The power $\alpha=1$ gives
\[
W_{n-1}=\frac{n(n-1)}2
\]
and therefore
\begin{equation}\label{eq:young-parent}
\Pp(K_n=i)=\frac{2i}{n(n-1)},
\qquad 1\le i<n.
\end{equation}
This is the young-age preferential tree of
\citet{LyonMahmoud2020}.  The special linear weight admits a particularly
simple exact generator.

\begin{proposition}[Two-sample parent rule]\label{prop:sampler}
Fix $n\ge2$.  Choose an ordered pair $(U,V)$ uniformly from
\[
\{(u,v):1\le u,v\le n,\ u\ne v\},
\]
and set
\begin{equation}\label{eq:sampler}
K_n:=\max\{U,V\}-1.
\end{equation}
Then $K_n$ has the distribution \eqref{eq:young-parent}.
\end{proposition}

\begin{proof}
For $1\le i<n$, the event $K_n=i$ is the event
$\max\{U,V\}=i+1$.  Exactly $2i$ ordered distinct pairs have maximum $i+1$:
either $U=i+1$ and $V\in\{1,\ldots,i\}$, or conversely.  Since there are
$n(n-1)$ ordered distinct pairs in total,
\[
\Pp(K_n=i)=\frac{2i}{n(n-1)}.
\]
\end{proof}

\paragraph{Implementation.}
An ordered distinct pair can be sampled without rejection.  Draw
$U\sim\mathrm{Unif}\{1,\ldots,n\}$ and
$R\sim\mathrm{Unif}\{1,\ldots,n-1\}$ independently, and set
\[
V=
\begin{cases}
R,&R<U,\\
R+1,&R\ge U.
\end{cases}
\]
Then apply \eqref{eq:sampler}.

\begin{corollary}[Generation complexity]\label{cor:generator}
In a unit-cost word-RAM model with $O(\log n)$-bit words and constant-time
uniform integer sampling, the parent array of a young-age preferential tree
of size $n$ can be generated exactly in $\Theta(n)$ time and $\Theta(n)$
space.  If the edges are streamed instead of stored, the extra working space
is $O(1)$.
\end{corollary}

The probabilistic results below do not rely on this sampler.  Its role is to
show that the distinguished young-age case also admits a direct randomized
generation algorithm with constant work per insertion.

\section{Expected total path length and the first-order phase transition}
\label{sec:expectation}

Throughout this section, $w_i=i^\alpha$ with fixed $\alpha\in\R$.

Let $U_n$ be uniform on $\{1,\ldots,n\}$, conditionally independent given the
tree.  If the operation \textsc{Trace-To-Root} repeatedly replaces a vertex by
its parent until the root is reached, then its number of pointer traversals is
$D_{U_n}$.  Conditional on the tree,
\begin{equation}\label{eq:conditional-average}
\E(D_{U_n}\mid T_n^{(\alpha)})=\frac{T_n}{n}.
\end{equation}
Thus $T_n/n$ is the quenched average query cost.

\begin{theorem}[Exact expectation and phase transition]
\label{thm:expectation}
For every $\alpha\in\R$ and $n\ge1$,
\begin{equation}\label{eq:ETexact}
\E T_n=\sum_{j=1}^{n-1}(n-j)p_j.
\end{equation}
Moreover:
\begin{enumerate}[label=(\roman*)]
\item If $\alpha>-1$ and $c_\alpha=\alpha+1$, then
\begin{equation}\label{eq:ET-log}
\frac{\E T_n}{n}
=
c_\alpha\log n+\kappa_\alpha+o(1),
\end{equation}
where
\begin{equation}\label{eq:kappa-alpha}
\kappa_\alpha
=
c_\alpha\gamma-c_\alpha
+
\sum_{j=1}^\infty
\left(p_j-\frac{c_\alpha}{j}\right).
\end{equation}

\item If $\alpha=-1$, then there exists a finite constant
$\kappa_{-1}$ such that
\begin{equation}\label{eq:ET-loglog}
\frac{\E T_n}{n}
=
\log\log n+\kappa_{-1}+o(1).
\end{equation}

\item If $\alpha<-1$, then
\begin{equation}\label{eq:ET-constant}
\frac{\E T_n}{n}\longrightarrow
\mu_\alpha,
\qquad
\mu_\alpha:=\sum_{j=1}^\infty p_j<\infty.
\end{equation}
\end{enumerate}
\end{theorem}

\begin{proof}
By Proposition~\ref{prop:bernoulli},
\[
\E T_n
=
\sum_{i=1}^n\E D_i
=
\sum_{i=2}^n\sum_{j=1}^{i-1}p_j.
\]
Interchanging the finite sums gives \eqref{eq:ETexact}.  Dividing by $n$,
\begin{equation}\label{eq:ETsplit}
\frac{\E T_n}{n}
=
\sum_{j=1}^{n-1}p_j
-
\frac1n\sum_{j=1}^{n-1}jp_j.
\end{equation}

Suppose first that $\alpha>-1$.  By Lemma~\ref{lem:p-asym},
\[
\sum_{j=1}^{n-1}p_j
=
c_\alpha H_{n-1}
+
\sum_{j=1}^\infty
\left(p_j-\frac{c_\alpha}{j}\right)
+o(1),
\]
while $jp_j\to c_\alpha$, so Cesàro's theorem gives
\[
\frac1n\sum_{j=1}^{n-1}jp_j\to c_\alpha.
\]
Using $H_{n-1}=\log n+\gamma+o(1)$ proves
\eqref{eq:ET-log}--\eqref{eq:kappa-alpha}.

For $\alpha=-1$, $p_j=1/(jH_j)$.  Since
\[
H_j=\log j+\gamma+O(j^{-1}),
\]
we have, for $j\ge2$,
\[
\frac1{jH_j}
=
\frac1{j\log j}
+
O\!\left(\frac1{j(\log j)^2}\right).
\]
The error is summable and the standard integral comparison gives
\[
\sum_{j=2}^{n-1}\frac1{j\log j}
=
\log\log n+C+o(1)
\]
for a finite constant $C$.  On the other hand,
$jp_j=1/H_j\to0$, so the second term of \eqref{eq:ETsplit} tends to zero.
This proves \eqref{eq:ET-loglog}.

Finally, if $\alpha<-1$, then $\sum_jp_j=\mu_\alpha<\infty$ and
$jp_j\to0$ by Lemma~\ref{lem:p-asym}.  Hence the first term in
\eqref{eq:ETsplit} tends to $\mu_\alpha$, while the second tends to zero by
Cesàro's theorem.
\end{proof}

The theorem can be read directly as an average-case complexity statement.

\begin{corollary}[Average root-traversal complexity]\label{cor:complexity}
The expected conditional average number of parent-pointer traversals satisfies
\[
\E\!\left[\E(D_{U_n}\mid T_n^{(\alpha)})\right]
\sim
\begin{cases}
(\alpha+1)\log n,&\alpha>-1,\\[1mm]
\log\log n,&\alpha=-1,\\[1mm]
\mu_\alpha,&\alpha<-1.
\end{cases}
\]
Thus $\alpha=-1$ separates logarithmic and constant-order average traversal
cost, with an iterated logarithm exactly at criticality.
\end{corollary}

\subsection{Second-order expectation inside the old-age phase}

The leading term for every $\alpha<-1$ is linear, but its deterministic
correction has another threshold at $\alpha=-2$.

\begin{proposition}\label{prop:old-second}
Let $\alpha<-1$.
\begin{enumerate}[label=(\roman*)]
\item If $-2<\alpha<-1$, then
\[
\E T_n
=
\mu_\alpha n
-
\frac{n^{\alpha+2}}
{\zeta(-\alpha)(-\alpha-1)(\alpha+2)}
(1+o(1)).
\]

\item If $\alpha=-2$, then
\[
\E T_n
=
\mu_{-2}n
-
\frac{\log n}{\zeta(2)}
+O(1).
\]

\item If $\alpha<-2$, then
\[
\E T_n
=
\mu_\alpha n-\lambda_\alpha+o(1),
\qquad
\lambda_\alpha:=\sum_{j\ge1}jp_j<\infty.
\]
\end{enumerate}
\end{proposition}

\begin{proof}
From \eqref{eq:ETexact},
\begin{equation}\label{eq:old-correction}
\mu_\alpha n-\E T_n
=
n\sum_{j\ge n}p_j
+
\sum_{j=1}^{n-1}jp_j.
\end{equation}
Write $c_\alpha:=1/\zeta(-\alpha)$.  By
Lemma~\ref{lem:p-asym},
\[
p_j\sim c_\alpha j^\alpha.
\]
If $-2<\alpha<-1$, the standard power-sum and regularly varying tail
estimates give
\[
n\sum_{j\ge n}p_j
\sim
\frac{c_\alpha}{-\alpha-1}\,n^{\alpha+2},
\qquad
\sum_{j=1}^{n-1}jp_j
\sim
\frac{c_\alpha}{\alpha+2}\,n^{\alpha+2}.
\]
Substitution in \eqref{eq:old-correction} yields
\[
\mu_\alpha n-\E T_n
\sim
\frac{n^{\alpha+2}}{\zeta(-\alpha)}
\left(
\frac1{-\alpha-1}+\frac1{\alpha+2}
\right)
=
\frac{n^{\alpha+2}}
{\zeta(-\alpha)(-\alpha-1)(\alpha+2)}.
\]

At the boundary $\alpha=-2$ we need one term more than the equivalence
$p_j\sim 1/(\zeta(2)j^2)$.  The tail of the Basel series satisfies
\[
W_j
=
\sum_{k=1}^j\frac1{k^2}
=
\zeta(2)-\frac1j+O(j^{-2}),
\]
and therefore
\[
\frac1{W_j}
=
\frac1{\zeta(2)}
+
\frac1{\zeta(2)^2j}
+
O(j^{-2}).
\]
Consequently,
\[
p_j
=
\frac1{\zeta(2)j^2}+O(j^{-3}).
\]
It follows that
\[
n\sum_{j\ge n}p_j=O(1),
\qquad
\sum_{j=1}^{n-1}jp_j
=
\frac1{\zeta(2)}\log n+O(1),
\]
which proves part~(ii).

Finally, if $\alpha<-2$, then
$\sum_{j\ge1}jp_j<\infty$.  Moreover,
\[
n\sum_{j\ge n}p_j
=
O(n^{\alpha+2})=o(1),
\]
while
\[
\sum_{j=1}^{n-1}jp_j
\longrightarrow
\lambda_\alpha.
\]
Equation~\eqref{eq:old-correction} now gives part~(iii).
\end{proof}

\subsection{The young-age mean and a first BST coincidence}

For $\alpha=1$,
\[
p_j=\frac{2}{j+1}.
\]
Equation~\eqref{eq:ETexact} then gives an exact closed form.

\begin{corollary}\label{cor:young-mean}
For the young-age preferential tree,
\begin{equation}\label{eq:young-mean}
\E T_n=2(n+1)H_n-4n.
\end{equation}
\end{corollary}

\begin{proof}
Substituting $p_j=2/(j+1)$ in \eqref{eq:ETexact} and summing gives
\[
2\sum_{j=1}^{n-1}\frac{n-j}{j+1}
=
2(n+1)H_n-4n.
\]
\end{proof}

The right-hand side of \eqref{eq:young-mean} is also the expected internal
path length of a random binary search tree on $n$ keys
\citep{Knuth1998,FlajoletSedgewick2009}.  Section~\ref{sec:youngage} will show
that the agreement stops at the first moment.

\section{Weighted-depth martingales and exact fluctuation identities}
\label{sec:martingale}

In this section the weights $w_i>0$ are arbitrary and deterministic; no growth assumption is imposed.
Let $\F_n$ be the sigma-field generated by the parent choices
$K_2,\ldots,K_n$.  Define the weighted depth sum
\begin{equation}\label{eq:An}
A_n:=\sum_{i=1}^n w_iD_i.
\end{equation}
By \eqref{eq:depth-rec},
\begin{equation}\label{eq:cond-depth}
\E(D_{n+1}\mid\F_n)
=
1+\frac{A_n}{W_n}.
\end{equation}
Introduce the innovation
\begin{equation}\label{eq:xi}
\xi_{n+1}
:=
D_{n+1}-\E(D_{n+1}\mid\F_n),
\qquad n\ge1.
\end{equation}
Then
\[
\E(\xi_{n+1}\mid\F_n)=0.
\]
We put $v_n:=\E\xi_n^2$ for $n\ge2$.

\begin{proposition}[Weighted-depth martingale]\label{prop:weighted-mg}
Let
\[
\widetilde A_n:=A_n-\E A_n.
\]
Then
\begin{equation}\label{eq:M-def}
M_n:=\frac{\widetilde A_n}{W_n}
=
\sum_{j=2}^n p_j\xi_j
\end{equation}
is a square-integrable martingale.
\end{proposition}

\begin{proof}
Using $A_{n+1}=A_n+w_{n+1} D_{n+1}$ and
\eqref{eq:cond-depth}--\eqref{eq:xi}, centering gives
\[
\widetilde A_{n+1}
=
\left(1+\frac{w_{n+1}}{W_n}\right)\widetilde A_n
+
w_{n+1}\xi_{n+1}.
\]
Since $W_{n+1}=W_n+w_{n+1}$,
\[
\frac{\widetilde A_{n+1}}{W_{n+1}}
=
\frac{\widetilde A_n}{W_n}
+
\frac{w_{n+1}}{W_{n+1}}\xi_{n+1}
=
M_n+p_{n+1}\xi_{n+1}.
\]
As $M_1=0$, iteration proves \eqref{eq:M-def}.
\end{proof}

The martingale in Proposition~\ref{prop:weighted-mg} is the derivative at zero of the normalized weighted-profile martingale described in the introduction.
We include its short direct proof to establish the notation needed for the exact innovation identities.

The same martingale gives an exact representation for total path length.

\begin{theorem}[Exact innovation representation]\label{thm:innovation}
For every $n\ge2$,
\begin{equation}\label{eq:Tinnovation}
T_n-\E T_n
=
\sum_{j=2}^n
\bigl(1+(n-j)p_j\bigr)\xi_j.
\end{equation}
Consequently,
\begin{equation}\label{eq:VarT-exact}
\Var(T_n)
=
\sum_{j=2}^n
\bigl(1+(n-j)p_j\bigr)^2v_j.
\end{equation}
\end{theorem}

\begin{proof}
From $T_{n+1}=T_n+D_{n+1}$ and \eqref{eq:cond-depth},
\[
(T_{n+1}-\E T_{n+1})-(T_n-\E T_n)
=
M_n+\xi_{n+1}.
\]
Summing from $1$ to $n-1$ and inserting
$M_m=\sum_{j=2}^mp_j\xi_j$ gives
\[
T_n-\E T_n
=
\sum_{j=2}^n\xi_j
+
\sum_{m=1}^{n-1}\sum_{j=2}^mp_j\xi_j
=
\sum_{j=2}^n\bigl(1+(n-j)p_j\bigr)\xi_j.
\]
Martingale differences are orthogonal in $L^2$, so
\eqref{eq:VarT-exact} follows.
\end{proof}

The innovation variances have a simple universal recurrence.

\begin{proposition}[Innovation variance recurrence]\label{prop:vrec}
Let
\[
V_n:=\Var(D_n)
=
\sum_{j=1}^{n-1}p_j(1-p_j).
\]
Then
\begin{equation}\label{eq:v-identity}
v_n
=
V_n-\sum_{j=2}^{n-1}p_j^2v_j
\end{equation}
and, for $n\ge2$,
\begin{equation}\label{eq:v-rec}
v_{n+1}
=
(1-p_n^2)v_n+p_n(1-p_n),
\qquad v_2=0.
\end{equation}
\end{proposition}

\begin{proof}
By the law of total variance,
\[
\Var(D_n)
=
\E\!\left[\Var(D_n\mid\F_{n-1})\right]
+
\Var\!\left(\E(D_n\mid\F_{n-1})\right).
\]
The first term is $v_n$.  By \eqref{eq:cond-depth} and
Proposition~\ref{prop:weighted-mg}, the second is
\[
\Var(M_{n-1})
=
\sum_{j=2}^{n-1}p_j^2v_j,
\]
which proves \eqref{eq:v-identity}. Subtracting consecutive identities and using
$V_{n+1}-V_n=p_n(1-p_n)$ gives \eqref{eq:v-rec}.
\end{proof}

\section{A linear-scale limit theorem for deterministic weights}
\label{sec:limit}

We first state the limit theorem for a class of positive deterministic
attachment weights.  Let $w_1,w_2,\ldots>0$, write
$W_n=\sum_{i=1}^n w_i$ and $p_n=w_n/W_n$, and assume that
\begin{equation}
  p_n=O(n^{-1}).
  \label{eq:general-weight-condition}
\end{equation}
All random variables are defined on the coupled sequence of trees obtained
from independent parent choices.  We use the innovations $\xi_n$, their
variances $v_n=\E\xi_n^2$, and the martingale
$M_n=\sum_{j=2}^n p_j\xi_j$ from the preceding section.

\begin{lemma}[Square summability for deterministic weights]
\label{lem:general-square-summability}
Under assumption~\eqref{eq:general-weight-condition},
\begin{equation}
  v_n=O(\log n),
  \qquad
  \sum_{j=2}^{\infty}p_j^2v_j<\infty,
  \qquad
  \sum_{j=2}^{\infty}\frac{v_j}{j^2}<\infty.
  \label{eq:general-square-summability}
\end{equation}
\end{lemma}

\begin{proof}
The innovation-variance recurrence gives
$v_{n+1}\le v_n+p_n$, with $v_2=0$.  Hence
\[
  v_n\le\sum_{j=2}^{n-1}p_j=O(\log n).
\]
Since $p_j=O(j^{-1})$, both series in
\eqref{eq:general-square-summability} are bounded by a constant multiple of
$\sum_{j=2}^{\infty}(\log j)/j^2$, apart from finitely many terms.
\end{proof}

\begin{theorem}[Path-length limit for deterministic weights]
\label{thm:general-path-length-limit}
Let $w_1,w_2,\ldots>0$ be deterministic attachment weights satisfying
\eqref{eq:general-weight-condition}.  Then the series
\begin{equation}
  Z:=\sum_{j=2}^{\infty}p_j\xi_j
  \label{eq:general-limit-series}
\end{equation}
converges almost surely and in $L^2$, and
\begin{equation}
  \frac{T_n-\E T_n}{n}\longrightarrow Z
  \qquad\text{almost surely and in }L^2.
  \label{eq:general-path-length-limit}
\end{equation}
Moreover, $\E Z=0$, the limit is nondegenerate, and
\begin{equation}
  \operatorname{Var}(T_n)\sim\sigma^2n^2,
  \qquad
  \sigma^2:=\E Z^2
    =\sum_{j=2}^{\infty}p_j^2v_j\in(0,\infty).
  \label{eq:general-variance-limit}
\end{equation}
\end{theorem}

\begin{proof}
By Lemma~\ref{lem:general-square-summability} and orthogonality of
martingale differences,
\[
  \sup_n\E M_n^2
  =\sum_{j=2}^{\infty}p_j^2v_j<\infty.
\]
The martingale convergence theorem therefore gives $M_n\to Z$ almost
surely and in $L^2$, with $\E Z=0$ and
$\E Z^2=\sum_{j=2}^{\infty}p_j^2v_j$.

The exact innovation representation in the preceding section yields
\begin{equation}
  \frac{T_n-\E T_n}{n}
  =M_n+\frac1n\sum_{j=2}^n a_j\xi_j,
  \qquad a_j:=1-jp_j.
  \label{eq:general-limit-remainder}
\end{equation}
Assumption~\eqref{eq:general-weight-condition} implies
$\sup_{j\ge2}|a_j|<\infty$.  Consequently,
\[
  \sum_{j=2}^{\infty}
    \E\left(\frac{a_j\xi_j}{j}\right)^2<\infty.
\]
The martingale series $\sum_{j=2}^{\infty}a_j\xi_j/j$ thus converges
almost surely.  Kronecker's lemma gives
\[
  \frac1n\sum_{j=2}^n a_j\xi_j\longrightarrow0
  \qquad\text{almost surely}.
\]
For convergence in $L^2$, orthogonality and
Lemma~\ref{lem:general-square-summability} give
\[
  \E\left(\frac1n\sum_{j=2}^n a_j\xi_j\right)^2
  =\frac1{n^2}\sum_{j=2}^n a_j^2v_j
  =O\left(\frac{\log n}{n}\right)\longrightarrow0.
\]
Together with $M_n\to Z$, these estimates prove
\eqref{eq:general-path-length-limit}.

Since all weights are strictly positive,
$0<p_2=w_2/(w_1+w_2)<1$ and $p_3>0$.  The recurrence with $v_2=0$
gives $v_3=p_2(1-p_2)>0$.  Therefore
\[
  \sigma^2\ge p_3^2p_2(1-p_2)>0,
\]
which proves nondegeneracy.  Finally, convergence in $L^2$ in
\eqref{eq:general-path-length-limit}, with both sides centered, implies
$\operatorname{Var}(T_n)/n^2\to\E Z^2=\sigma^2$ and hence
\eqref{eq:general-variance-limit}.
\end{proof}

\begin{corollary}[Power-weight path-length limit]
\label{cor:power-path-length-limit}
For every fixed $\alpha\in\mathbb R$, let $w_j=j^\alpha$ and
$p_j=j^\alpha/W_j$.  Then
\begin{equation}
  Z_\alpha:=\sum_{j=2}^{\infty}p_j\xi_j,
  \qquad
  \frac{T_n-\E T_n}{n}\longrightarrow Z_\alpha
  \quad\text{almost surely and in }L^2,
  \label{eq:power-path-length-limit}
\end{equation}
where the series converges almost surely and in $L^2$, and
\begin{equation}
  \operatorname{Var}(T_n)\sim\sigma_\alpha^2n^2,
  \qquad
  \sigma_\alpha^2:=\E Z_\alpha^2
    =\sum_{j=2}^{\infty}p_j^2v_j\in(0,\infty).
  \label{eq:power-variance-limit}
\end{equation}
\end{corollary}

\begin{proof}
For $\alpha>-1$, the power-sum estimate gives
$p_j\sim(\alpha+1)/j$.  At $\alpha=-1$,
$p_j=1/(jH_j)\le1/j$.  For $\alpha<-1$, the bound $W_j\ge w_1=1$
gives $p_j\le j^\alpha\le1/j$.  Thus
\eqref{eq:general-weight-condition} holds in every case, and the result
follows from Theorem~\ref{thm:general-path-length-limit}.
\end{proof}

\begin{remark}[No regular cumulative-weight asymptotic is needed]
For $w_n=1+\frac12\sin(\log n)$ we have $\frac12\le w_n\le\frac32$, hence $p_n\le3/n$.
The sum--integral error is $O(\log n)$ because the absolute derivative of $\sin(\log x)$ is at most $1/x$.
Integration therefore gives
\[
\frac{W_n}{n}=1+\frac14\bigl(\sin(\log n)-\cos(\log n)\bigr)
+O\!\left(\frac{\log n}{n}\right).
\]
Thus $W_n/n$ has no limit, and $W_n\sim c n^\gamma$ fails for every $c,\gamma>0$.
Theorem~\ref{thm:general-path-length-limit} nevertheless applies.
Thus no polynomial cumulative-weight asymptotic is needed.
\end{remark}

For the remainder of the paper we specialize again to $w_i=i^\alpha$.

\subsection{Recovery of the classical uniform recursive tree}
\label{subsec:uniform}

The case $\alpha=0$ is the classical uniform recursive tree.  It provides a
useful exact benchmark for the general theory.

\begin{corollary}[Uniform recursive tree benchmark]\label{cor:uniform}
Let $\alpha=0$.  Then $p_j=1/j$ and, for every $n\ge2$,
\begin{equation}\label{eq:uniform-mean}
\E T_n=n(H_n-1).
\end{equation}
Moreover,
\begin{equation}\label{eq:uniform-martingale}
\frac{T_n-\E T_n}{n}
=
\sum_{j=2}^n\frac{\xi_j}{j},
\end{equation}
and
\begin{equation}\label{eq:uniform-var-exact}
\frac{\Var(T_n)}{n^2}
=
2-H_n^{(2)}-\frac{H_n}{n},
\qquad
H_n^{(2)}:=\sum_{k=1}^n\frac1{k^2}.
\end{equation}
Consequently,
\[
\frac{T_n-n(H_n-1)}{n}
\longrightarrow Z_0
\qquad\text{almost surely and in }L^2,
\]
and
\begin{equation}\label{eq:uniform-var-asym}
\Var(T_n)
\sim
\left(2-\frac{\pi^2}{6}\right)n^2.
\end{equation}
\end{corollary}

\begin{proof}
The expectation follows from Theorem~\ref{thm:expectation}:
\[
\E T_n
=
\sum_{j=1}^{n-1}\frac{n-j}{j}
=
nH_{n-1}-(n-1)
=
n(H_n-1).
\]
Since $p_j=1/j$, the coefficient in
Theorem~\ref{thm:innovation} simplifies to
\[
1+(n-j)p_j
=
1+\frac{n-j}{j}
=
\frac{n}{j},
\]
which proves \eqref{eq:uniform-martingale}.

For the innovation variances, recurrence~\eqref{eq:v-rec} becomes
\[
v_{n+1}
=
\left(1-\frac1{n^2}\right)v_n
+
\frac1n\left(1-\frac1n\right),
\qquad v_2=0.
\]
A direct induction gives
\begin{equation}\label{eq:uniform-v}
v_n=\frac{n}{n-1}H_{n-1}-2,
\qquad n\ge2.
\end{equation}
By \eqref{eq:uniform-martingale} and orthogonality,
\[
\frac{\Var(T_n)}{n^2}
=
\sum_{j=2}^n\frac{v_j}{j^2}.
\]
Applying identity~\eqref{eq:v-identity} with $n+1$ yields
\[
\sum_{j=2}^n\frac{v_j}{j^2}
=
\Var(D_{n+1})-v_{n+1}.
\]
Here Proposition~\ref{prop:bernoulli} and \eqref{eq:uniform-v} give
\[
\Var(D_{n+1})=H_n-H_n^{(2)},
\qquad
v_{n+1}=\frac{n+1}{n}H_n-2.
\]
Subtraction proves \eqref{eq:uniform-var-exact}.  Finally,
$H_n^{(2)}\to\pi^2/6$ and $H_n/n\to0$, while the asserted almost-sure and
$L^2$ convergence is the specialization of
Corollary~\ref{cor:power-path-length-limit}.
\end{proof}

The limiting result for the uniform recursive tree is classical
\citep{Mahmoud1991,DobrowFill1999}.  Corollary~\ref{cor:uniform} is included
to show that the general innovation framework specializes exactly to this
benchmark; no novelty is claimed for the classical limit itself.

Combining Theorem~\ref{thm:expectation} and
Corollary~\ref{cor:power-path-length-limit} gives the pathwise average-case complexity.

\begin{corollary}[Pathwise phase transition]\label{cor:pathwise}
The following convergences hold almost surely and in $L^2$:
\[
\begin{array}{lll}
\alpha>-1:
&
\dfrac{T_n}{n}-(\alpha+1)\log n
\longrightarrow
\kappa_\alpha+Z_\alpha,
\\[4mm]
\alpha=-1:
&
\dfrac{T_n}{n}-\log\log n
\longrightarrow
\kappa_{-1}+Z_{-1},
\\[4mm]
\alpha<-1:
&
\dfrac{T_n}{n}
\longrightarrow
\mu_\alpha+Z_\alpha.
\end{array}
\]
In particular,
\[
\frac{T_n}{n\log n}\to\alpha+1
\quad\text{a.s. for }\alpha>-1,
\]
and
\[
\frac{T_n}{n\log\log n}\to1
\quad\text{a.s. for }\alpha=-1.
\]
\end{corollary}

\begin{remark}
The limit theorem distinguishes two levels of behavior.  The dominant
complexity scale is deterministic: logarithmic, iterated logarithmic, or
constant.  After removing that scale, however, an order-one random correction
survives in the average cost, equivalently an order-$n$ random correction
survives in total path length.
\end{remark}

\section{The summable old-age phase}\label{sec:oldage}

When $\alpha<-1$, the total attachment weight is finite:
\[
W_\infty
:=
\sum_{i=1}^\infty i^\alpha
=
\zeta(-\alpha).
\]
This permits a more concrete interpretation of the limit in
Corollary~\ref{cor:pathwise}.

\begin{theorem}[Random constant-time limit]\label{thm:old-random}
Let $\alpha<-1$ and define
\[
A_\infty
:=
\sum_{i=1}^\infty i^\alpha D_i.
\]
Then $A_\infty<\infty$ almost surely and
\begin{equation}\label{eq:old-pathwise}
\frac{T_n}{n}
\longrightarrow
1+\frac{A_\infty}{\zeta(-\alpha)}
\end{equation}
almost surely and in $L^2$.
Moreover,
\[
Z_\alpha
=
\frac{A_\infty-\E A_\infty}{\zeta(-\alpha)}
\]
and
\[
\mu_\alpha
=
1+\frac{\E A_\infty}{\zeta(-\alpha)}.
\]
\end{theorem}

\begin{proof}
For $\alpha<-1$, Proposition~\ref{prop:bernoulli} and
Lemma~\ref{lem:p-asym} imply $\sup_i\E D_i<\infty$.  Hence
\[
\E\left[\sum_{i=1}^\infty i^\alpha D_i\right]
=
\sum_{i=1}^\infty i^\alpha\E D_i
<\infty.
\]
Thus $A_\infty<\infty$ almost surely and $A_n\uparrow A_\infty$ in $L^1$.

Since $W_n\to\zeta(-\alpha)$,
\[
M_n
=
\frac{A_n-\E A_n}{W_n}
\longrightarrow
\frac{A_\infty-\E A_\infty}{\zeta(-\alpha)}
\]
in probability.  But Corollary~\ref{cor:power-path-length-limit} gives
$M_n\to Z_\alpha$ almost surely and in $L^2$, so the two limits agree.

Next, from \eqref{eq:cond-depth},
\[
\E D_{n+1}
=
1+\frac{\E A_n}{W_n}.
\]
The left side tends to
$\sum_{j\ge1}p_j=\mu_\alpha$, while the right side tends to
$1+\E A_\infty/\zeta(-\alpha)$.  Therefore
\[
\mu_\alpha
=
1+\frac{\E A_\infty}{\zeta(-\alpha)}.
\]
Substitution in Corollary~\ref{cor:pathwise} gives
\eqref{eq:old-pathwise}.  The $L^2$ convergence follows from that corollary.
\end{proof}

Thus the old-age regime has constant average traversal complexity but is not
self-averaging: early random attachment decisions remain visible in the
limiting cost.

\section{Exact second-order analysis for young-age preference}
\label{sec:youngage}

We return to $\alpha=1$.  In this case
\[
p_j=\frac{2}{j+1}.
\]
\citet{LyonMahmoud2020} showed that the insertion depth has the same
distribution as the insertion depth in a random binary search tree and, in
particular,
\begin{equation}\label{eq:young-depth-var}
\Var(D_n)
=
2(H_n+1)-4H_n^{(2)},
\end{equation}
where
\[
H_n^{(2)}:=\sum_{k=1}^n\frac1{k^2}.
\]

The martingale innovations make it possible to determine the global
path-length variance constant explicitly.

\begin{proposition}[Exact innovation variance]\label{prop:young-v}
For the young-age tree and every $n\ge2$,
\begin{equation}\label{eq:young-v}
v_n
=
\frac{2(n+1)(n+2)}{n(n-1)}
\left(
H_n-3+\frac6{n+2}
\right).
\end{equation}
Consequently,
\[
v_n=2H_n-6+o(1).
\]
\end{proposition}

\begin{proof}
For $p_n=2/(n+1)$, recurrence \eqref{eq:v-rec} becomes
\[
v_{n+1}
=
\left(1-\frac4{(n+1)^2}\right)v_n
+
\frac{2(n-1)}{(n+1)^2},
\qquad v_2=0.
\]
Direct substitution shows that \eqref{eq:young-v} satisfies this recurrence
and the initial condition.  The asymptotic statement follows from
$H_n=\log n+\gamma+o(1)$.
\end{proof}

\begin{theorem}[Young-age path-length variance]\label{thm:young-var}
For every $n\ge2$,
\begin{equation}\label{eq:young-var-exact}
\Var(T_n)
=
\sum_{j=2}^n
\left(\frac{2n-j+1}{j+1}\right)^2
v_j,
\end{equation}
where $v_j$ is given by \eqref{eq:young-v}.  Furthermore,
\begin{equation}\label{eq:young-var-asym}
\Var(T_n)
\sim
\left(8-\frac{2\pi^2}{3}\right)n^2.
\end{equation}
\end{theorem}

\begin{proof}
Since $p_j=2/(j+1)$,
\[
1+(n-j)p_j
=
1+\frac{2(n-j)}{j+1}
=
\frac{2n-j+1}{j+1}.
\]
Equation~\eqref{eq:young-var-exact} is therefore the specialization of
\eqref{eq:VarT-exact}.

For the limit constant, identity \eqref{eq:v-identity} gives
\[
\sum_{j=2}^{n-1}p_j^2v_j
=
\Var(D_n)-v_n.
\]
Using \eqref{eq:young-depth-var} and Proposition~\ref{prop:young-v},
\[
\Var(D_n)-v_n
=
\bigl(2H_n+2-4H_n^{(2)}\bigr)
-
\bigl(2H_n-6+o(1)\bigr).
\]
Since $H_n^{(2)}\to\pi^2/6$,
\[
\sigma_1^2
=
\sum_{j=2}^\infty p_j^2v_j
=
8-\frac{2\pi^2}{3}.
\]
Corollary~\ref{cor:power-path-length-limit} now gives
\eqref{eq:young-var-asym}.
\end{proof}

Numerically,
\[
8-\frac{2\pi^2}{3}
=
1.4202637326\ldots.
\]

\subsection{Same mean as a random BST, different fluctuations}

Let $P_n^{\mathrm{BST}}$ be the internal path length of a random binary search
tree constructed by inserting $n$ distinct keys in uniformly random order.
Classical analysis gives
\begin{align}
\E P_n^{\mathrm{BST}}
&=
2(n+1)H_n-4n,\label{eq:bst-mean}\\
\Var(P_n^{\mathrm{BST}})
&\sim
\left(7-\frac{2\pi^2}{3}\right)n^2.
\label{eq:bst-var}
\end{align}
See, for example, \citet{Knuth1998} and
\citet{FlajoletSedgewick2009}.

Combining \eqref{eq:young-mean}, \eqref{eq:young-var-asym},
\eqref{eq:bst-mean}, and \eqref{eq:bst-var} yields the following.

\begin{corollary}[Local equality, global separation]\label{cor:bst}
For every $n$,
\[
\E T_n=\E P_n^{\mathrm{BST}},
\]
but
\[
\frac{\Var(T_n)}{n^2}
-
\frac{\Var(P_n^{\mathrm{BST}})}{n^2}
\longrightarrow1.
\]
Thus the young-age tree and random BST have identical expected total path
length but distinct linear-scale fluctuations.
\end{corollary}

\begin{center}
\begin{tabular}{@{}lcc@{}}
\toprule
 & Young-age tree & Random BST\\
\midrule
Insertion-depth law & same as BST & same as young-age\\
$\E$ total path length
& $2(n+1)H_n-4n$
& $2(n+1)H_n-4n$\\
Asymptotic variance coefficient
& $8-\dfrac{2\pi^2}{3}$
& $7-\dfrac{2\pi^2}{3}$\\
\bottomrule
\end{tabular}
\end{center}

The equality of insertion-depth marginals therefore hides a difference in
the dependence structure across insertion times.  The total path length,
being a sum of strongly dependent depths, exposes this difference at second
order.

\section{Root-traversal costs}\label{sec:alg}

The results can be stated directly in terms of a simple online parent-pointer
data structure.  Each insertion stores one parent pointer.  Consider the
operation

\medskip
\noindent
\begin{minipage}{0.92\textwidth}
\ttfamily
TRACE-TO-ROOT(v)\\
\phantom{xx}cost $\leftarrow 0$\\
\phantom{xx}while v $\neq$ root do\\
\phantom{xxxx}v $\leftarrow$ parent[v]\\
\phantom{xxxx}cost $\leftarrow$ cost + 1\\
\phantom{xx}return cost
\end{minipage}
\medskip

For a uniformly sampled stored vertex $U_n$, the conditional expected
number of pointer traversals is $T_n/n$ by \eqref{eq:conditional-average}.  Hence
Corollary~\ref{cor:pathwise} gives not only an annealed expected complexity
but a pathwise statement about a typical realized data structure.

\begin{corollary}\label{cor:query}
For a realized power-weight recursive tree,
\[
\E(D_{U_n}\mid T_n^{(\alpha)})
=
\begin{cases}
(\alpha+1)\log n+\kappa_\alpha+Z_\alpha+o(1),
&\alpha>-1,\\[1mm]
\log\log n+\kappa_{-1}+Z_{-1}+o(1),
&\alpha=-1,\\[1mm]
\mu_\alpha+Z_\alpha+o(1),
&\alpha<-1,
\end{cases}
\]
almost surely.
\end{corollary}

There are two complementary messages.  First, the deterministic leading
complexity changes sharply at $\alpha=-1$.  Second, the random tree-to-tree
variation in the average query cost remains of constant order after the
leading term is removed.  Equivalently, total path-length fluctuations remain
of order $n$ in every regime.

The young-age case adds a generation result: by
Corollary~\ref{cor:generator}, a tree with the same nonuniform parent law can
be sampled in linear total time without maintaining a table of cumulative
weights.  Thus both construction and subsequent root-traversal costs admit
direct average-case analysis.

\section{Discussion and open problems}\label{sec:discussion}
The innovation identities apply to every positive deterministic weight sequence.
The condition $p_n=O(n^{-1})$ is sufficient for a nondegenerate linear-scale limit and includes power weights on both sides of their summability threshold.
For $\alpha>-1$, the almost-sure limit can also be recovered from existing weighted-profile theory.
The present argument gives a direct common treatment that additionally includes the critical, summable, and irregular cumulative-weight cases, and yields explicit variance identities.

The coefficient $1+(n-j)p_j$ describes how the innovation introduced at time $j$ contributes to total path length at time $n$.
Its form explains why fluctuations of order $n$ persist even when the mean depth remains bounded.
For linear weights, evaluating these identities gives a global distinction from random binary search trees despite equality of the insertion-depth marginals.

Several questions remain open for this approach.
\begin{enumerate}[label=(\roman*)]
\item The condition $p_n=O(n^{-1})$ is sufficient, but need not be necessary.
It would be useful to characterize more general deterministic weights that admit a nondegenerate limit on the same scale.
\item The distribution of $Z_\alpha$ is represented here by a martingale series.
Much more is known in the uniform case \citep{Mahmoud1991,DobrowFill1999}; fixed-point descriptions or higher-moment recurrences for general power weights would complement the present analysis.
\item The regularity, monotonicity, and asymptotic behavior of $\alpha\mapsto\sigma_\alpha^2$ deserve separate study.
\item Other global functionals, such as the Wiener index or profile width, may further distinguish linear-weight recursive trees from random binary search trees.
\end{enumerate}

\section{Conclusion}\label{sec:conclusion}
We gave an innovation-based analysis of total path length for deterministic weighted recursive trees.
The exact representation and variance recurrence lead to almost-sure and $L^2$ convergence of centered total path length divided by $n$ whenever $p_n=O(n^{-1})$.
The limiting variance is finite and strictly positive.
This provides a common elementary treatment without a polynomial cumulative-weight asymptotic.

For power weights, the theorem applies to every real exponent.
It complements the existing profile-based limit theory for $\alpha>-1$ and covers the critical and summable regimes $\alpha\le-1$.
In the summable regime the limiting average depth is expressed through the weighted depths of the infinite tree.
For linear weights the variance coefficient is $8-2\pi^2/3$, exceeding that of a random binary search tree by exactly one although all insertion-depth marginals agree.
The same results describe random corrections to the average cost of tracing parent pointers to the root.

\section*{Data availability}
No data were used in this study.

\section*{Funding}
This research did not receive any specific grant from funding agencies in the
public, commercial, or not-for-profit sectors.

\section*{Declaration of competing interest}
The authors declare no competing interests.

\section*{AI-assisted manuscript revision}
During revision, ChatGPT (OpenAI) was used for language editing, literature comparison,
and assistance in checking and presenting mathematical arguments.
The authors take full responsibility for the content of the manuscript.

% Bibliography intentionally embedded in this source file.

\end{document}